\documentclass[a4,11pt]{article}
\usepackage{amsmath,amsfonts,amssymb}
\usepackage{graphicx}
\usepackage[english,activeacute]{babel}
\usepackage[latin1]{inputenc}

\newtheorem{theorem}{Theorem}

\newtheorem{corollary}{Corollary}
\newtheorem{lemma}{Lemma}
\newtheorem{remark}{Remark}
\newenvironment{proof}[1][Proof]{\noindent\textbf{#1.} }{\ \rule{0.5em}{0.5em}}
\begin{document}

\title{Strong and ratio asymptotics for Laguerre polynomials revisited}
\author{Alfredo Dea\~{n}o$^{1,\dag}$, Edmundo J. Huertas$^{2}$ and Francisco Marcell\'an$^{2}$
\thanks{%
The three authors has been supported by Direcci\'on General de Investigaci\'on,
Ministerio de Econom\'ia y Competitividad of Spain, grant MTM2012-36732-C03-01. A. Dea\~{n}o acknowledges financial support from Ministerio de Ciencia e Innovaci\'on of Spain, project MTM2009--11686, and from the Fund for Scientific Research Flanders, project G.0617.10. $\dag $ Corresponding author.} \\
$^1$Department of Computer Science, KU Leuven\\
Celestijnenlaan 300A, 3001 Heverlee, Belgium\\
$^2$Departamento de Matem\'aticas, Universidad Carlos III de Madrid\\
Avenida de la Universidad 30, 28911, Legan\'es, Spain\\[3mm]
alfredo.deano@cs.kuleuven.be$^{\dag }$, ehuertasce@gmail.com\\
pacomarc@ing.uc3m.es}
\date{\today}
\maketitle

\begin{abstract}
In this paper we consider the strong asymptotic behavior of Laguerre polynomials in the complex plane. The leading behavior is well known from Perron and Mehler--Heine formulas, but higher order coefficients, which are important in the context of Krall-Laguerre or Laguerre-Sobolev-type orthogonal polynomials, are notoriously difficult to compute. In this paper, we propose the use of an alternative expansion, due to Buchholz, in terms of Bessel functions of the first kind. The coefficients in this expansion can be obtained in a straightforward way using symbolic computation. As an application, we derive extra terms in the asymptotic expansion of ratios of Laguerre polynomials in $\mathbb{C}\setminus[0,\infty)$.
%
\end{abstract}




\section{Introduction}

The classical Laguerre polynomials $L_n^{(\alpha)}(x)$ are defined as the polynomials orthogonal with
respect to the $L^2$ inner product

\begin{equation*}
\left\langle p,q\right\rangle _{\alpha }=\int_{0}^{\infty }p(x)q(x)x^{\alpha
}e^{-x}dx, \qquad\alpha >-1,\quad p,q\in \mathbb{P},
\end{equation*}
see for instance \cite{Chi78} or \cite{Szego75}. We normalize them according to \cite{Szego75}:
\begin{equation}
L_n^{(\alpha)}(x)=\frac{(-1)^n}{n!}\hat{L}_n^{(\alpha)}(x),
\end{equation}
where $\hat{L}_n^{(\alpha)}(x)$ denotes the monic Laguerre polynomial of degree $n$. 


These polynomials can be given in terms of an $\,_{1}F_{1}$ confluent
hypergeometric function as follows (see, for
instance, \cite{DLMF}, \cite{Szego75}):

\begin{equation}
L_{n}^{(\alpha)}(x)={\binom{n+\alpha }{n}}\,_{1}F_{1}(-n;\alpha +1;x).
\label{Lag1F1}
\end{equation}



The asymptotic behavior of Laguerre polynomials in the complex plane as $n\to\infty$ is well known in the literature. In particular, we recall the following two asymptotic expansions:

Outer strong asymptotics (in $\mathbb{C}\setminus\mathbb{R}_{+}$) is given by Perron's formula.
For $\alpha>-1$ we get%
\begin{equation}\label{Perron}
L_{n}^{(\alpha )}\left( z\right) =\frac{1}{2\sqrt{\pi}}e^{z/2}\left(
-z\right) ^{-\alpha /2-1/4}n^{\alpha /2-1/4}e^{2\sqrt{ -nz}}\left\{ \sum\limits_{m=0}^{d-1}C_{m}(z)n^{-m/2}+\mathcal{O}%
(n^{-d/2})\right\} .
\end{equation}

Here $C_{k}(z)$ is independent of $n$, and $C_0(k)=1$. This relation holds for $z$\ in the
complex plane with a cut along the positive real semiaxis. The bound for the
remainder holds uniformly in every closed domain of the complex plane with
empty intersection with $\mathbb{R}_{+}$ (see \cite{Szego75}, Theorem
8.22.3).

Mehler--Heine type formula. Fixed $j\in
\mathbb{Z}^+$ and $J_{\alpha }$\ the Bessel function of
the first kind, then%
\begin{equation}  \label{MH}
\lim_{n\rightarrow \infty }\frac{L_{n}^{(\alpha )}\left( z/(n+j)\right) }{%
n^{\alpha }}=z^{-\alpha /2}J_{\alpha }\left( 2\sqrt{z}\right) ,
\end{equation}%
uniformly over compact subsets of $\mathbb{C}$ (see \cite{Szego75}, Theorem
8.1.3).

Higher order coefficients in the asymptotic expansion \eqref{Perron} are important when one deals with Krall--Laguerre or Laguerre--Sobolev--type orthogonal polynomials. More precisely, they play a key role in the analysis of their outer relative asymptotics, see \cite{DHM1d-NA11}. There one needs to estimate ratios of Laguerre
orthogonal polynomials like%
\begin{equation*}
\frac{L_{n+j}^{(\alpha )}(x)}{L_{n}^{(\beta )}(x)},
\end{equation*}%
where $n=0,1,2,\ldots $, $j\in \mathbb{Z}$. Additionally, we require $\alpha
,\beta >-1$. Note that in the particular case where $j$ is an integer and $\alpha=\beta$, this asymptotic information can in principle be obtained from the three--term recurrence relation for Laguerre polynomials, applying the Perron theorem, see for example \cite[\S 4.3]{GST2007}, but for general values of $\alpha$, $\beta$ and $j$ this procedure is not feasible.

We consider the Laguerre kernel
\begin{equation*}
K_{n-1}(x,y)=\frac{n+1}{||L_{n-1}^{(\alpha )}||_{\alpha }^{2}}\frac{%
L_{n-1}^{(\alpha )}(x)L_{n}^{(\alpha )}(y)-L_{n}^{(\alpha
)}(x)L_{n-1}^{(\alpha )}(y)}{x-y},
\end{equation*}
and we denote by
\begin{equation*}
K_{n-1}^{(j,k)}(x,y)=\frac{\partial ^{j+k}K_{n-1}(x,y)}{\partial x^{j}\partial y^{k}}%
\end{equation*}%
their partial derivatives of order $j$ and $k$ with respect to the variables $x$ and $y,$ respectively.

If we compute the $d$-th derivative with respect to the second
variable and we evaluate it at $y=c$, that is $K_{n}^{(0,d)}(x,c)$, we find expressions of the form
\begin{equation}
\sum_{j=0}^{d}\binom{d}{j}(-1)^{j+1}\frac{L_{n-(1+d)+j}^{(\alpha +d)}\left(
c\right) }{L_{n-(1+d)}^{(\alpha +d)}(c)}.  \label{Lag-Prob1}
\end{equation}%

For instance, when $d=1$ we get
\begin{equation}
\begin{aligned}
& K_{n-1}^{(0,1)}(x,c)=\frac{n!(n-1)!}{||\hat{L}^{(\alpha)}_{n-1}||_{\alpha}^2}L_{n-1}^{(\alpha)}(x)L_{n-2}^{(\alpha+1)}(c)\\
&\times\left[\frac{1}{(x-c)^2}\left(\frac{L_{n-1}^{(\alpha+1)}(c)}{L_{n-2}^{(\alpha+1)}(c)}-1\right)\left(\frac{L_{n}^{(\alpha)}(c)}{L_{n-1}^{(\alpha)}(c)}-\frac{L_{n}^{(\alpha)}(x)}{L_{n-1}^{(\alpha)}(x)}\right)+\frac{1}{x-c}\left(\frac{L_{n}^{(\alpha)}(x)}{L_{n-1}^{(\alpha)}(x)}-\frac{L_{n-1}^{(\alpha+1)}(c)}{L_{n-2}^{(\alpha+1)}(c)}\right)
\right],
\end{aligned}
\end{equation}
see \cite{DHM1d-NA11}. When $d=2$, we get
\begin{equation}
\begin{aligned}
&K_{n-1}^{(0,2)}(x,c)=\frac{n\cdot 2!}{||L_{n-1}^{(\alpha )}||_{\alpha }^{2}}%
\cdot L_{n-3}^{(\alpha +2)}(c)L_{n-1}^{(\alpha )}\left( x\right)\\
&\times \left [ \frac{1}{(x-c)^{3}} \left( \frac{L_{n-1}^{(\alpha +2)}\left( c\right) }{%
L_{n-3}^{(\alpha +2)}(c)}-2\frac{L_{n-2}^{(\alpha +2)}\left( c\right) }{%
L_{n-3}^{(\alpha +2)}(c)}+1\right)\left( \frac{%
L_{n}^{(\alpha )}(c)}{L_{n-1}^{(\alpha )}(c)}-\frac{L_{n}^{(\alpha )}(x)}{%
L_{n-1}^{(\alpha )}(x)}\right) \right.  \\
&\left.\hspace{3mm} - \frac{1}{(x-c)^{2}}\left( \frac{L_{n-2}^{(\alpha +2)}(c)}{L_{n-3}^{(\alpha +2)}(c)}%
-1\right)\left( \frac{L_{n-1}^{(\alpha +1)}(c)}{%
L_{n-2}^{(\alpha +1)}(c)}-\frac{L_{n}^{(\alpha )}(x)}{L_{n-1}^{(\alpha )}(x)}%
\right)+\frac{1}{2!(x-c)}\left( \frac{L_{n-2}^{(\alpha +2)}(c)}{%
L_{n-3}^{(\alpha +2)}(c)}-\frac{L_{n}^{(\alpha )}(x)}{L_{n-1}^{(\alpha )}(x)}%
\right) \right ].
\end{aligned}
\end{equation}

These formulas follow from differentiation and using two well known identities for Laguerre polynomials. The first one is the so called structure relation
\begin{equation}
L_{n}^{(\alpha )}(x)=L_{n}^{(\alpha +1)}\left( x\right) -L_{n-1}^{(\alpha
+1)}\left( x\right), \qquad n\geq 1.
\end{equation}
The second one is the Hahn property
\begin{equation}
\lbrack L_{n}^{(\alpha )}(x)]^{\prime }=-L_{n-1}^{(\alpha +1)}(x), \qquad n\geq 1,
\label{MVAH-022}
\end{equation}
see for example \cite{Szego75}, formulas (5.1.13) and (5.1.14).

When considering the asymptotic behavior for $n$ large enough, one needs higher order coefficients in $n$ for those terms of the form \eqref{Lag-Prob1} in the previous expression, because of cancellation. More precisely, we need to know exactly the coefficient of $n^{-d/2}$ to estimate (\ref{Lag-Prob1}) correctly. For example, if $d=1$ we need to know the coefficient of $n^{-1/2}$, if $d=2$ the coefficient of $n^{-1}$,
and so on. Thus, expressions like
\begin{equation*}
\lim_{n\rightarrow \infty }n^{(\ell -j)/2}\frac{L_{n+k}^{(\alpha +j)}(x)}{%
L_{n+h}^{(\alpha +\ell )}(x)}=(-x)^{(\ell -j)/2},\quad j,\ell \in \mathbb{R}%
,\quad h,k\in \mathbb{Z},
\end{equation*}%
appearing in  \cite[\S 8]{Alvare04}, or Lemmas 1 and 2 in \cite{DHM1d-NA11}, are not accurate enough if $d\geq 3$.

The main advantage of Perron's expansion for Laguerre polynomials is the
simplicity of the asymptotic sequence (inverse powers of $n$), but it has
the problem that the coefficients $C_k(z)$ soon become cumbersome to
compute. The standard expansion, valid for $z$ outside of the support of the orthogonality measure, is due to
Perron \cite{Perron1921}, based on ideas of the steepest descent method, and then it
appears referenced in \cite{Szego75} and many others.

One possibility to derive higher order coefficients is to use the generating function for Laguerre polynomials:
\begin{equation*}
(1-z)^{-\alpha -1}\text{exp}\left( \frac{xz}{z-1}\right) =\sum_{m=0}^{\infty
}L_{m}^{(\alpha )}(x)z^{m},\qquad |z|<1,
\end{equation*}
then write the coefficients as contour integrals and apply the standard method of steepest descent, see
for instance \cite{BH2010} or \cite{Olver1974}. However, the computations soon become complicated, since parametrizing the path of steepest descent is not easy in explicit form.

In principle, it is  also possible to compute higher order terms in
Perron's expansion using Riemann--Hilbert techniques. Following ideas of \cite{DKMVZ}, in \cite{Vanlessen}
the author presents strong asymptotics of
orthogonal polynomials with respect to the weight $w(x)=e^{-Q(x)}$ on $%
[0,\infty)$, where $Q(x)$ is a general polynomial with positive leading coefficient. This clearly includes the
Laguerre case. In order to obtain higher order terms, it is necessary to
work further into the last matrix of the nonlinear steepest descent method, $%
R(z)$, and then point out the details undoing the transformations.

To the best of our knowledge, the only sources of information for higher order coefficients in the Perron expansion are the works by W. Van Assche \cite{VanAss01,ErrVanAss01}, and  D. Borwein, J. M. Borwein and R. E. Crandall in \cite{Borwein1}. This last one is based on complex integral representations with strict error bounds. It provides a powerful method to generate the coefficients $C_m$.

In this paper we propose an alternative approach, based solely on using an expansion of the Laguerre polynomials that involves more complicated special functions, namely Bessel functions of the first kind. This type of expansions go back to the works of Tricomi and Buchholz, see \cite[\S 7.4]{Buch}, and also \cite{Sla}, \cite[\S 10.15]{Bat2} and references therein. In this way, the different behaviors of $%
L^{(\alpha)}_{n}(x) $ in the complex plane are better captured, and thus the coefficients are simpler. Moreover, apart from the large $n$ asymptotic behavior, the resulting approximation converges in the complex plane. Reexpanding the Bessel functions in inverse powers of $n$ destroys this property, but using straightforward manipulations it is possible to recover the results in \cite{Borwein1}.

The structure of the manuscript is as follows. In Section 2 we analyze some expansions of confluent hypergeometric functions in terms of Bessel functions of the first kind. Thus we deduce  an algorithm which gives the polynomial coefficients of  such expansions for Laguerre polynomials. This approach has many advantages with respect to the standard Perron expansion. Based on this method, in Section 3 we deduce the asymptotic expansions of outer ratio asymptotics for Laguerre polynomials which cover higher order terms than those used in the literature until now.

\section{Tricomi and Buchholz expansions in terms of Bessel functions}
In \cite{LT}, the authors study several expansions of confluent
hypergeometric functions in terms of Bessel functions, introduced by Tricomi and
Buchholz. Some of these expansions are useful in an asymptotic sense. We
have
\begin{equation}
_{1}F_{1}(a;c;z)=\Gamma (c)e^{z/2}\sum_{m=0}^{\infty }\left( \frac{z}{2}%
\right) ^{m}A_{m}(a,c)E_{c-1+m}(\kappa z),  \label{1F1Bessel}
\end{equation}%
where
\begin{equation}\label{kappa}
\kappa =\frac{c}{2}-a
\end{equation}
is the Whittaker parameter, the function $%
E_{\alpha }(z)$ can be expressed in terms of the standard Bessel function of the first kind
\begin{equation}
E_{\alpha }(z)=z^{-\alpha /2}J_{\alpha }(2\sqrt{z}),  \label{Enu}
\end{equation}%
and the coefficients $A_{n}(a,c)$ satisfy the recursion
\begin{equation*}
(m+1)A_{m+1}=(m+c-1)A_{m-1}-2\kappa A_{m-2},\qquad m\geq 2,
\end{equation*}%
with initial values
\begin{equation*}
A_{0}(a,c)=1,\qquad A_{1}(a,c)=0,\qquad A_{2}(a,c)=\frac{c}{2},\qquad
A_{3}(a,c)=-\frac{1}{3}(c-2a).
\end{equation*}

The expansion \eqref{1F1Bessel} is convergent in the entire $z$ plane. This
follows from results of Tricomi, see \cite{Sla} and \cite{TemTricomi}.

For Laguerre polynomials $L_{n}^{(\alpha )}(z)$, we set $a=-n$
and $c=\alpha +1$, i.e.%
\begin{equation*}
L_{n}^{(\alpha )}(z)={\binom{n+\alpha }{n}}\,_{1}F_{1}(-n;\alpha +1;z).
\end{equation*}

Observe that for positive real values of $z$, the Bessel functions $%
J_{\alpha}(2\sqrt{\kappa z})$ are oscillatory, a fact that is consistent with the
behavior of the Laguerre polynomial in the interval of orthogonality. Also,
because of the construction of the function $E_{\alpha}(z)$, the series expansion is
also asymptotic for large values of $\kappa$. Indeed, $\kappa=n+\frac{%
\alpha+1}{2}$, so the argument of the Bessel function becomes large when the
degree of the Laguerre polynomial $n$ grows. Note also that if we take the
first term in this expansion and scale the variable $z$, we recover the
Mehler--Heine asymptotics \eqref{MH}. However, we point out an important
drawback of this expansion: the coefficients $A_n(a,c)$ are polynomials in $%
a $, and hence they grow with $n$ in the Laguerre case.

An alternative expansion given in \cite{LT} is due to Buchholz, see \cite[\S
7.4]{Buch}, and has a similar structure:
\begin{equation}
_{1}F_{1}(a;c;z)=\Gamma (c)e^{z/2}\sum_{m=0}^{\infty }\left( \frac{z}{2}%
\right) ^{m}P_{m}(c,z)E_{c-1+m}(\kappa z),  \label{1F1BesselBuch}
\end{equation}%
where $\kappa $ is introduced in \eqref{kappa} and $E_{c-1+m}(\kappa z)$ is given in %
\eqref{Enu}. The new coefficients $P_{m}(c,z)$ satisfy a more complicated
relation:
\begin{equation}\label{recBuch}
P_{m}(c,z)=z^{-m/2}\int_{0}^{z}\left( \frac{1}{4}%
uP_{m-1}(c,u)+(c-2)P_{m-1}^{\prime }(c,u)-uP_{m-1}^{\prime \prime
}(c, u)\right) u^{m/2-1}du
\end{equation}%
for $m\geq 1$, with initial value $P_{0}(c,z)=1$. The first coefficients are
\begin{equation*}
\begin{aligned} P_1(c,z)&=\frac{z}{6}, \quad
P_2(c,z)=\frac{1}{72}(z^2+12c-24), \quad
P_3(c,z)&=\frac{z}{6480}(5z^2+180c-432). \end{aligned}
\end{equation*}

An important advantage of this expansion is that the coefficients do no longer
depend on $a$, so they remain bounded in the Laguerre case as $n\to \infty$.
Actually, as explained in \cite{LT}, the expansion has an asymptotic
character for large values of $\kappa$ (i.e. large values of $n$) with $%
\kappa z>0$, when $|z|=o(\kappa)$.

Thus, we obtain
\begin{equation}  \label{LagBessel}
\begin{aligned} L_n^{(\alpha)}(z)&= {n+\alpha \choose n} \Gamma(\alpha+1)
e^{z/2} \sum_{m=0}^{\infty} \left(\frac{z}{2}\right)^m P_m(\alpha+1,z)
E_{m+\alpha}(\kappa z)\\ &=\frac{\Gamma(n+\alpha+1)}{n!}\,e^{z/2}
\sum_{m=0}^{\infty} \left(\frac{z}{2}\right)^m P_m(\alpha+1,z) (\kappa
z)^{-\frac{m+\alpha}{2}} J_{m+\alpha}(2\sqrt{\kappa z})\\
&=\frac{\Gamma(n+\alpha+1)}{n!}\,e^{z/2}(\kappa z)^{-\alpha/2}
\sum_{m=0}^{\infty} \left(\frac{z}{4\kappa}\right)^{m/2} P_m(\alpha+1,z)
J_{m+\alpha}(2\sqrt{\kappa z}). \end{aligned}
\end{equation}

Again it is possible to recover the Mehler--Heine asymptotics from
this expression.


The aim of this section is to explain how to construct an asymptotic
expansion in terms of negative powers of $n$ from \eqref{LagBessel} in a systematic
way. The coefficients become fairly complicated as well, but the procedure
is easily implemented using symbolic computation.


In order to rewrite \eqref{LagBessel} in terms of negative powers of $n$, we use the
asymptotic approximation for the Bessel function of large argument, \cite[%
10.17.3]{DLMF}:
\begin{equation*}
J_{\alpha }(z)\sim \left( \frac{2}{\pi z}\right) ^{{\frac{1}{2}}}\left( \cos
\omega \sum_{k=0}^{\infty }(-1)^{k}\frac{a_{2k}(\alpha )}{z^{2k}}-\sin
\omega \sum_{k=0}^{\infty }(-1)^{k}\frac{a_{2k+1}(\alpha )}{z^{2k+1}}\right)
,\quad |\arg \,z|<\pi ,
\end{equation*}%
where
\begin{equation}\label{omega}
\omega =\omega (\alpha )=z-\frac{\alpha \pi }{2}-\frac{\pi }{4}
\end{equation}%
and the coefficients are
\begin{equation}\label{coeffsa}
a_{k}(\alpha )=\frac{(4\alpha ^{2}-1)(4\alpha ^{2}-9)\ldots (4\alpha
^{2}-(2k-1)^{2})}{8^{k}k!},\qquad k\geq 1,
\end{equation}
$a_{0}(\alpha )=1$.
Consequently,
\begin{equation*}
J_{m+\alpha }(2\sqrt{\kappa z})\sim \pi^{-1/2} (\kappa z)^{-1/4} \left( \cos
\omega \sum_{k=0}^{\infty }(-1)^k\frac{a_{2k}(m+\alpha )}{(4\kappa z)^k}-\sin
\omega \sum_{k=0}^{\infty }(-1)^k\frac{a_{2k+1}(m+\alpha )}{(4\kappa z)^{k+1/2}}\right),
\end{equation*}%
for $|\arg(\kappa z)|<2\pi$ and integer $m\geq 0$. Using these expressions, we obtain the following result:

\begin{theorem}
For $\alpha>-1$, the Laguerre polynomial $L_n^{(\alpha)}(z)$ admits the following asymptotic expansion as $n\to
\infty$:
\begin{equation}\label{goal}
\begin{aligned}
L_{n}^{(\alpha )}(z)&= \frac{\Gamma (n+\alpha +1)}{n!}\,\pi^{-1/2}\,e^{z/2}(\kappa
z)^{-\alpha /2-1/4}\\
&\times \left[\sum_{m=0}^{d}\frac{B_{2m}(\alpha ,z)\cos\omega}{n^{m}}+
\sum_{m=0}^{d}\frac{B_{2m+1}(\alpha ,z)\sin\omega}{n^{m+1/2}}+\mathcal{O}(n^{-d-1})\right],
\end{aligned}
\end{equation}
for some coefficients $B_{m}(\alpha,z)$ independent of $n$. The error term holds uniformly for $z$ in compact sets of $\mathbb{C}$, and the parameter $\omega=\omega(\alpha)$ is given by \eqref{omega}.
\end{theorem}

Here, the coefficients $B_{m}(\alpha ,z)$  come from the expansion of the Bessel
functions. We observe that we have a sum of Bessel functions of different
orders, but the sines and cosines that appear in the asymptotic expansions
can be somehow grouped together, since
\begin{equation}
\begin{aligned} \cos\left(z-\frac{(\alpha+m)\pi}{2}-\frac{\pi}{4}\right)
&=(-1)^s\begin{cases} \cos\left(z-\frac{\alpha\pi}{2}-\frac{\pi}{4}\right),
\qquad m=2s,\\ \sin\left(z-\frac{\alpha\pi}{2}-\frac{\pi}{4}\right), \qquad
m=2s+1, \end{cases}\\
\sin\left(z-\frac{(\alpha+m)\pi}{2}-\frac{\pi}{4}\right)
&=(-1)^s\begin{cases} \sin\left(z-\frac{\alpha\pi}{2}-\frac{\pi}{4}\right),
\qquad m=2s,\\ -\cos\left(z-\frac{\alpha\pi}{2}-\frac{\pi}{4}\right), \qquad
m=2s+1, \end{cases}. \end{aligned}  \label{shiftalpha}
\end{equation}
for $s=0,1,2,\ldots$



An extra care is needed in two aspects when assembling the expansion, though:
first, one has to take into account the different $\pm\cos\omega$ and $%
\pm\sin\omega$ factors that multiply the asymptotic expansion of the Bessel
functions; second, the terms $a_{2k}(\alpha)$ and $a_{2k+1}(\alpha)$ depend
on $\alpha$, so they change at each level.

In the sequel, let us fix an integer $d\geq 1$, and let us group the different terms of order $n^{-d/2}$
that come into play in the sum \eqref{LagBessel}:
\begin{itemize}
\item If $d=2M$ is even, then the term corresponding to $m=0$ multiplies the
Bessel function $J_{\alpha}(2\sqrt{\kappa z})$, that we have to expand up to
order $\kappa^{-d}$, so this term is multiplied by $(-1)^d
\cos\omega(\alpha) $. The term corresponding to $m=2$ will be multiplied by $%
J_{\alpha+2}(2\sqrt{\kappa z})$, that we need to expand up to order $%
\kappa^{-N+1}$, because of the $(z/4\kappa)^{m/2}$ factor. Taking into account
the shift properties of the cosine, this is multiplied by $%
(-1)^{M-1}\cos\omega(\alpha+2)=(-1)^d\cos\omega(\alpha)$.  Similarly, all the
even terms up to $m=d$ will have the same common factor.

The term corresponding to $m=1$ is multiplied by $(-1)^{d}\sin\omega(%
\alpha+1)=(-1)^{d-1}\cos\omega(\alpha)$. For $m=3$, we have $%
(-1)^{d+1}\sin\omega(\alpha+3)=(-1)^{d+1}\cos\omega(\alpha)$, and so on.
Consequently, we have an alternating sum of the form
\begin{equation}  \label{Seven}
\begin{aligned} S_{2M}(\alpha,z)&=(-1)^{M}\sum_{m=0}^{2M}
(-1)^m \left(\frac{z}{4\kappa}\right)^{m/2}
P_m(\alpha+1,z)\frac{a_{2M-m}(\alpha+m)}{(2\sqrt{\kappa z})^{2M-m}}\\
&=(-4\kappa z)^{-M}\sum_{m=0}^{2M} (-1)^m
z^{m}P_m(\alpha+1,z)a_{2M-m}(\alpha+m). \end{aligned}
\end{equation}
for $M=0,1,\ldots$

\item If $d=2M+1$, then the term corresponding to $m=0$ multiplies
the Bessel function $J_{\alpha}(2\sqrt{\kappa z})$, that we have to expand
up to order $\kappa^{-d}$, so this term is multiplied by $(-1)^{M+1}
\sin\omega(\alpha)$. The term corresponding to $m=2$ will be multiplied by $%
J_{\alpha+2}(2\sqrt{\kappa z})$, that we need to expand up to order $%
\kappa^{-d+1}$, and because of the shift properties of the sine, this is
multiplied by $-(-1)^{M-1}\sin\omega(\alpha+2)=(-1)^{M-1}\sin\omega(\alpha)$. %
Similarly, all the even terms up to $m=d$ will have the same common factor.

For $m=1$ we have a factor $(-1)^{M}\cos\omega(\alpha+1)
=(-1)^{M}\sin\omega(\alpha)$. For $m=3$ we obtain $(-1)^{M-1}\cos\omega(%
\alpha+3)=(-1)^{M}\sin\omega(\alpha)$, and so on.

Thus,
\begin{equation}  \label{Sodd}
\begin{aligned} S_{2M+1}(\alpha,z)&=(-1)^{M+1}
\sum_{m=0}^{2M+1} (-1)^m \left(\frac{z}{4\kappa}\right)^{m/2}
P_m(\alpha+1,z)\frac{a_{2M+1-m}(\alpha+m)}{(2\sqrt{\kappa z})^{2M+1-m}}\\
&=(-1)^{M+1}(4\kappa z)^{-M-1/2}\sum_{m=0}^{2M+1} (-1)^m
z^{m} P_m(\alpha+1,z) a_{2M+1-m}(\alpha+m) \end{aligned}
\end{equation}
for $M=0,1,\ldots$
\end{itemize}

Observe that these expansions are in negative powers of $\kappa$, which is
essentially $n$. A further step is needed to retain only the terms up to
order $n^{-d/2}$ in each case, but this is simple using symbolic computation.

From \eqref{Seven} and \eqref{Sodd}, it is possible to compute the coefficients $B_m(\alpha,z)$ in \eqref{goal}. Schematically, one needs to do the following steps:
\begin{enumerate}
\item Fix the maximum order $d$.
\item Generate the polynomials $P_m(\alpha,z)$ for $m\leq d$, using the recursion \eqref{recBuch}.
\item Compute the coefficients of the asymptotic expansion of the Bessel functions from \eqref{coeffsa} up to order $d$.
\item Compute $S_m(\alpha,z)$ for $m\leq d$, sum all these terms, expand in inverse powers of $n$ and truncate.
\end{enumerate}

As an example we will find the first few coefficients. Using \textsc{Maple}, we obtain
\begin{equation*}
\begin{aligned} B_0(\alpha,z)&=1,\\
B_1(\alpha,z)&=\frac{4z^2-12\alpha^2+3}{48\sqrt{z}},\\
B_2(\alpha,z)&=-\frac{z^3}{288}+\frac{4\alpha^2+11}{192}z-\frac{(4\alpha^2-1)(4\alpha^2-9)}{512z}\\
B_3(\alpha,z)&=-\frac{z^{9/2}}{10368}+\frac{20%
\alpha^2+187}{23040}z^{5/2}-\frac{\alpha+1}{48}z^{3/2}-\frac{(4\alpha^2-9)(4\alpha^2-25)}{6144}z^{1/2}\\
&-\frac{(
\alpha+1)(4\alpha^2-1)}{64z^{1/2}}
+\frac{(4\alpha^2-1)(4\alpha^2-9)(4\alpha^2-25)}{24576 z^{3/2}}\\
B_4(\alpha,z)&=\frac{z^6}{497664}-\frac{20\alpha^2+391}{829440}z^4 +\frac{\alpha+1}{576}z^3
+\frac{80\alpha^4-584\alpha^2+9261}{737280}z^2\\
& -\frac{(\alpha+1)(4\alpha^2+11)}{384}z
-\frac{(2\alpha+1)(2\alpha+3)(2\alpha+5)(2\alpha+7)(4\alpha^2-32\alpha+27)}{294912}\\
&+\frac{(\alpha+1)(4\alpha^2-1)(4\alpha^2-9)}{1024z}
+\frac{(4\alpha^2-1)(4\alpha^2-9)(4\alpha^2-25)(4\alpha^2-49)}{1572864z^2}.
\end{aligned}
\end{equation*}

Higher order terms become complicated but follow the very same idea. This expansion presents two essential advantages with respect to the classical one
given by Perron. First, the coefficients are still complicated but they can be
computed systematically, up to the accuracy desired. Secondly, the original
expansion is convergent on the whole complex plane. So, retaining the Bessel
functions instead of expanding them in negative powers of $n,$ it provides a
useful representation of the Laguerre polynomials for large degree.

One difficulty of the previous expansion is that it contains $\cos\omega$ and $\sin\omega$ terms. These terms can be grouped together away from $[0,\infty)$. Note that
\begin{equation}
\cos \omega =\frac{1}{2}\left[ e^{i(2\sqrt{\kappa z}-\frac{\alpha \pi
}{2}-\frac{\pi }{4})}+e^{-i(2\sqrt{\kappa z}-\frac{\alpha \pi }{2}-%
\frac{\pi }{4})}\right]
 \end{equation}

If $\text{Im}\,z<0$, then the first exponential term dominates. On the other hand, if $%
\text{Im}\,z>0,$ then the second exponential term dominates. We can write
this as follows:
\begin{equation*}
\cos \omega \sim \frac{1}{2}e^{2\sqrt{-\kappa z}}e^{\pm i\left(
\frac{\alpha \pi }{2}+\frac{\pi }{4}\right)}, \qquad \pm \textrm{Im}\,z>0,
\end{equation*}%
with the understanding that $-z=ze^{\mp \pi i}$ if $\pm \textrm{Im}\,z>0$, i.e. we take the principal argument throughout. In a similar way, we can write
%
\begin{equation*}
\sin \omega \sim \mp \frac{1}{2i} e^{2\sqrt{-\kappa z}}e^{\pm i\left(
\frac{\alpha \pi }{2}+\frac{\pi }{4}\right)}, \qquad \pm \textrm{Im}\,z>0,
\end{equation*}

Now, it is not difficult to check that we can write
\begin{equation}
\begin{aligned} z^{-\alpha/2-1/4}\cos\omega&\sim
\frac{1}{2}(-z)^{-\alpha/2-1/4}e^{2\sqrt{-\kappa z}},\\
z^{-\alpha/2-1/4}\sin\omega&\sim
\pm\frac{i}{2}(-z)^{-\alpha/2-1/4}e^{2\sqrt{-\kappa z}}, \qquad \pm\,
\textrm{Im}\, z>0, \end{aligned}  \label{pm}
\end{equation}%
again with the same criterion to understand $-z$ in the complex plane. So we
can pull out a common factor in the whole expansion:

\begin{theorem}
Let $\alpha>-1$. The Laguerre polynomial $L_n^{(\alpha)}(z)$ admits the following asymptotic expansion as $n\to
\infty$:
\begin{equation}\label{Laguerreexp}
L_{n}^{(\alpha )}(z)=\frac{1}{2\sqrt{\pi}} \frac{\Gamma (n+\alpha +1)}{n!}\,
e^{z/2}(-\kappa z)^{-\alpha
/2-1/4} e^{2\sqrt{-\kappa z}}\left[\sum_{m=0}^{d-1}\hat{B}_m(\alpha,z)n^{-m/2}+\mathcal{O}(n^{-d/2})\right],
\end{equation}%
for some coefficients $\hat{B}_{m}(\alpha,z)$ independent of $n$. The error term is uniform for $z$ in bounded sets of $\mathbb{C}\setminus [0,\infty)$ and the coefficients $\hat{B}_{m}(\alpha ,z)$ are related to the original ones $B_{m}(\alpha,z)$ in the following way:
\begin{equation}
\begin{aligned}
\hat{B}_{2m}(\alpha,z)&=B_{2m}(\alpha,z)\\
\hat{B}_{2m+1}(\alpha,z)&=\pm i B_{2m+1}(\alpha,z), \qquad \pm \textrm{Im}\, z>0.
\end{aligned}
\end{equation}
\end{theorem}

It is important to take into account that this modified expansion is only valid for $z$ in bounded sets of  $\mathbb{C}\setminus[0,\infty)$. On the positive real axis it is not possible to pull out a common factor from the sines and cosines, since the argument is real and thus both exponential terms are oscillatory and comparable in size.

\begin{remark}\label{rem1}
When $z<0$ we consider $\arg\, z=\pi$, so $-z=ze^{-\pi i}$.
\end{remark}
\begin{remark}\label{rem2}
It is possible to remove the $\pm$ in the complex plane by writing the odd coefficients in terms of the variable $-z$ and using the fact that $z^{m/2}=(-z)^{m/2}e^{\pm m\pi i/2}$, with $\pm \text{Im}\,z>0$ and $m\in\mathbb{Z}$. For instance,
\begin{equation}
 \hat{B}_1(\alpha,z)=\frac{4z^2-12\alpha^2+3}{48\sqrt{-z}},\\
\end{equation}

This extra step compensates the different sign in $\hat{B}_{2m+1}(\alpha,z)$ in the upper and lower half--plane, but one needs to check each power of $z$ within $B_{2m+1}(\alpha,z)$ individually.
\end{remark}

It is not complicated to recover the standard Perron expansion from the previous one. If we compare \eqref{Laguerreexp} with \eqref{Perron}, we observe that we just need to expand the terms in the prefactor in inverse powers of $n$ and combine them with the coefficients $\hat{B}_m(\alpha,z)$. In this way we have
\begin{equation}
\begin{aligned}
C_0(\alpha,z)&=\hat{B}_0(\alpha,z),\\
C_1(\alpha,z)&=\frac{\sqrt{-z}(\alpha+1)}{2}\hat{B}_0(\alpha,z)+\hat{B}_1(\alpha,z)\\
&=\frac{4z^2-24(\alpha+1)z-12\alpha^2+3}{48\sqrt{-z}},\\
C_2(\alpha,z)&=-\frac{(\alpha+1)(1-2\alpha+z(\alpha+1))}{8}\hat{B}_0(\alpha,z)+\frac{\sqrt{-z}(\alpha+1)}{2}\hat{B}_1(\alpha,z)+\hat{B}_2(\alpha,z)\\
&=-\frac{z^3}{288}+\frac{(\alpha+1)z^2}{24}-\frac{(20\alpha^2+48\alpha+13)z}{192}-\frac{(2\alpha-1)(2\alpha-3)(\alpha+1)}{32}
-\frac{(4\alpha^2-1)(4\alpha^2-9)}{512z}.
\end{aligned}
\end{equation}

Observe that $C_1(\alpha,z)$ coincides with the known coefficient in the
Perron expansion, see \cite{VanAss01} and \cite{ErrVanAss01}. The next coefficient $C_2(\alpha,z)$ agrees with the one  given in \cite[p. 3306]{Borwein1}, making the changes $a\mapsto -\alpha$ and $z\mapsto -z$ and reexpanding the prefactor therein, which is written in terms of the shifted variable $m=n+1$.

\section{Asymptotics of ratios of Laguerre polynomials}


As an application, the above expansion that we have obtained allows us to determine the outer asymptotic
behavior as $n\rightarrow \infty $ of arbitrary ratios of Laguerre
polynomials with greater accuracy than the formulas available in the
literature. We write the previous expansion as follows:
\begin{equation*}
L_{n}^{(\alpha )}(z)=f^{(\alpha)}_n(z)\left( \sum_{m=0}^{d-1}\hat{B}_{m}(\alpha
,z)n^{-m/2}+\mathcal{O}(n^{-d/2})\right) ,
\end{equation*}%
with the prefactor
\begin{equation}\label{fn}
f^{(\alpha)}_n(z)=\frac{1}{2\sqrt{\pi}}\frac{\Gamma (n+\alpha +1)}{n!}\,e^{z/2}(-\kappa(n,\alpha
)z)^{-\alpha/2-1/4} e^{2\sqrt{-nz}}.
\end{equation}%

Note that we emphasize that $\kappa$ depends both on $n$ and $\alpha$, since we want to consider different degree and different parameter of the Laguerre polynomials. Thus
\begin{equation*}
\begin{aligned} \frac{L_{n+j}^{(\alpha)}(z)}{L_{n}^{(\beta )}(z)}
&=\frac{f_{n+j}^{(\alpha)}(z)}{f^{(\beta)}_n(z)}\sum_{k=0}^{d-1}
D_k(\alpha,\beta,z)n^{-k/2}+\mathcal{O}(n^{-d/2}), \end{aligned}
\end{equation*}%
where
\begin{equation*}
\begin{aligned} \frac{f^{(\alpha)}_{n+j}(z)}{f^{(\beta)}_n(z)}
&=(-z)^{\frac{\beta-\alpha}{2}}\frac{\Gamma(n+j+\alpha+1)}{\Gamma(n+j+1)}\frac{%
\Gamma(n+1)}{\Gamma(n+\beta+1)}\frac{\kappa(n,\beta)^{\beta/2+1/4}}{\kappa(%
n+j,\alpha)^{\alpha/2+1/4}}\, e^{2\sqrt{-(\kappa+j)z}-2\sqrt{-\kappa z}}
\end{aligned}
\end{equation*}

Notice that we need different coefficients $\hat{B}_k(\alpha,z)$ in the
numerator, as a result of using degree $n+j$ instead of $n$. However, this
is easily computed with the same \textsc{Maple} procedure, setting $%
\kappa=n+j+\frac{\alpha+1}{2}$.

The $D_k(\alpha,\beta,z)$ coefficients can be computed in an automatic way. For instance, the first ones are
\begin{equation}
\begin{aligned}
D_0(\alpha,\beta,z)&=1,\\
D_1(\alpha,\beta,z)&=\pm i \frac{\beta^2-\alpha^2}{4\sqrt{z}}, \qquad \pm\textrm{Im}\, z>0,\\
D_2(\alpha,\beta,z)&=\frac{(\beta^2-\alpha^2)(3\alpha^2+9\beta^2-4z^2-9)}{96z}.\\
\end{aligned}
\end{equation}

Bearing in mind that $z^{-1/2}=\mp i (-z)^{-1/2}$ when $\pm \textrm{Im}\,z>0$, so we can write
\begin{equation}
D_1(\alpha,\beta,z)=\frac{\beta^2-\alpha^2}{4\sqrt{-z}}.
\end{equation}

This is the type of modification mentioned in Remark \ref{rem2} before, that consists of rewriting the odd coefficients in terms of $-z$ instead of $z$. It becomes more complicated to implement for higher terms.


The prefactor can be expanded in negative powers of $n$ as a convolution of two series. Firstly, the ratio of exponential functions is straightforward to compute symbolically. Then

\begin{lemma} We have
\begin{equation*}
\frac{\kappa (n,\beta )^{\beta /2+1/4}}{\kappa (n+j,\alpha )^{\alpha /2+1/4}}=n^{%
\frac{\beta -\alpha }{2}}\sum_{m=0}^{\infty }A_{m}(j,\alpha ,\beta )n^{-m}, \qquad n\to\infty,
\end{equation*}%
where for $m\geq 0$ 
\begin{equation}\label{Am}
A_{m}(j,\alpha ,\beta )=\left( \frac{2j+\alpha +1}{2}\right)
^{m}\sum_{k=0}^{m}(-1)^{m-k}{\binom{\frac{\beta }{2}+\frac{1}{4}}{k}}{\binom{%
m-k+\frac{\alpha }{2}-\frac{3}{4}}{m-k}}\left( \frac{\beta +1}{2j+\alpha +1}%
\right) ^{k}.
\end{equation}
\end{lemma}
\begin{proof}
We note that
\begin{equation}
\begin{aligned}
\frac{\kappa (\beta )^{\beta /2+1/4}}{\kappa (\alpha )^{\alpha /2+1/4}}&=
n^{\frac{\beta-\alpha}{2}}\left(1+\frac{\beta+1}{2n}\right)^{\beta/2+1/4} \left(1+\frac{2j+\alpha+1}{2n} \right)^{-\alpha/2-1/4}\\
&=n^{\frac{\beta-\alpha}{2}}
\sum_{k=0}^{\infty} {\frac{\beta}{2}+\frac{1}{4}\choose k}\left(\frac{\beta+1}{2n}\right)^{k}
\times
\sum_{m=0}^{\infty} {m+\frac{\alpha}{2}-\frac{3}{4}\choose m}(-1)^{m} \left(\frac{2j+\alpha+1}{2n}\right)^{m}\\
&=n^{\frac{\beta-\alpha}{2}}
\sum_{m=0}^{\infty} \sum_{k=0}^m {\frac{\beta}{2}+\frac{1}{4}\choose k}\left(\frac{\beta+1}{2n}\right)^{k}
{m-k+\frac{\alpha}{2}-\frac{3}{4}\choose m-k}(-1)^{m-k} \left(\frac{2j+\alpha+1}{2n}\right)^{m-k}\\
&=n^{\frac{\beta-\alpha}{2}}
\sum_{m=0}^{\infty}A_m(j,\alpha,\beta)n^{-m},
\end{aligned}
\end{equation}
where the coefficients $A_m(j,\alpha,\beta)$ are given by \eqref{Am}.
\end{proof}

The ratio of Gamma functions can also be expanded as follows:
\begin{lemma}
We have
\begin{equation*}
\frac{\Gamma (n+j+\alpha +1)}{\Gamma (n+\beta +1)}\sim n^{j+\alpha -\beta
}\sum_{m=0}^{\infty }\frac{G_{m}(\alpha ,\beta ,j)}{n^{m}}, \qquad n\to\infty,
\end{equation*}%
where the coefficients $G_{m}(\alpha ,\beta ,j)$ can be expressed in terms
of generalized Bernoulli polynomials:
\begin{equation*}
G_{m}(\alpha ,\beta ,j)={\binom{j+\alpha -\beta }{m}}B_{m}^{(j+\alpha -\beta
+1)}(j+\alpha +1).
\end{equation*}
\end{lemma}
\begin{proof} See formulas \cite[5.11.13, 5.11.17]{DLMF}.
\end{proof}

Notice that, unfortunately, these generalized Bernoulli polynomials are not directly implemented in \textsc{Maple}. However, they are the coefficients of the following generating function:
\begin{equation}
\left(\frac{t}{e^t-1}\right)^{\ell} e^{xt}=\sum_{n=0}^{\infty} B_n^{(\ell)}(x)\frac{t^n}{n!}, \qquad |t|<2\pi,
\end{equation}
see \cite[24.16.1]{DLMF}. It is straightforward to compute this expansion symbolically, select the coefficients and construct the generalized Bernoulli polynomials from there.

As a consequence,
\begin{corollary} We have
\begin{equation*}
\begin{aligned}
\frac{\Gamma (n+j +1)}{\Gamma (n+1)}
&\sim n^{j}\sum_{m=0}^{\infty }\frac{G_{m}(0,0,j)}{n^{m}}=n^{j}\sum_{m=0}^{\infty }{j\choose m}\frac{B_{m}^{(j+1)}(j+1)}{n^{m}}, \qquad n\to\infty.
\end{aligned}
\end{equation*}%
\end{corollary}

Therefore, this ratio can be computed as a special case of the previous one, setting $\alpha=\beta=0$.\\

Putting together all the previous computations, we arrive at a general formula for the ratio asymptotic of Laguerre polynomials of arbitrary degree and parameter, as $n\to\infty$:
\begin{theorem}
Let $\alpha,\beta>-1$, $j\in\mathbb{R}$ and $z\in\mathbb{C}\setminus [0,\infty)$, and fix an integer $d\geq 1$, then the ratio of arbitrary Laguerre polynomials has the following asymptotic expansion as $n\to\infty$:
\begin{equation}
\frac{L_{n+j}^{(\alpha)}(z)}{L_{n}^{(\beta)}(z)}
=\left(-\frac{z}{n}\right)^{\frac{\beta-\alpha}{2}}\sum_{m=0}^{d-1} U_m(\alpha,\beta,j,z)n^{-m/2}+\mathcal{O}(n^{-d/2}),
\end{equation}
where the first coefficients are
\begin{equation}
\begin{aligned}
U_0(\alpha,\beta,j,z)&=1,\\
U_1(\alpha,\beta,j,z)&=\frac{\beta^2-\alpha^2+2z(\beta-\alpha-2j)}{4\sqrt{-z}},\\
U_2(\alpha,\beta,j,z)&=-\frac{(6j^2+6(\alpha-\beta)j+\alpha^2+2\beta^2-3\alpha\beta) z}{12}
+\frac{(\beta^2-\alpha^2+2\alpha-1)j}{4}\\
&+\frac{(\alpha^2-\beta^2-2\alpha-2\beta-1)(\beta-\alpha)}{8}
-\frac{(\alpha^2+3\beta^2-3)(\alpha^2-\beta^2)}{32z}
\end{aligned}
\end{equation}

The error term holds uniformly for $z$ in compact sets of $\mathbb{C}\setminus[0,\infty)$.

\end{theorem}

As particular cases, we get the following results:

\begin{itemize}
 \item If we set $\beta=\alpha+1$ and $j=0$, we obtain
\begin{equation}
 \frac{L_{n}^{(\alpha)}(z)}{L_{n}^{(\alpha+1)}(z)} \sim \sqrt{\frac{-z}{n}}+\left(\frac{\alpha}{2}+\frac{1}{4}+\frac{z}{2}\right)\frac{1}{n}+\mathcal{O}(n^{-3/2}),
\end{equation}
and if $\beta=\alpha+2$ and $j=0$,
\begin{equation}
 \frac{L_{n}^{(\alpha)}(z)}{L_{n}^{(\alpha+2)}(z)} \sim \frac{-z}{n}+\frac{\sqrt{-z}\,( z+\alpha+1)}{n^{3/2}}+\mathcal{O}(n^{-2}),
\end{equation}
consistently with \cite[\S 3, Lemma 2]{DHM1d-NA11}.

 \item If we set $\beta=\alpha$, then
 \begin{equation}
 U_{0}(\alpha,\alpha,j,z)=1, \qquad U_{1}(\alpha,\alpha,j,z)=\sqrt{-z}j, \qquad
 U_{2}(\alpha,\alpha,j,z)=\frac{(2\alpha-2jz-1)j}{4},
\end{equation}
 so
\begin{equation}
 \frac{L_{n+j}^{(\alpha)}(z)}{L_{n}^{(\alpha)}(z)} \sim 1+\frac{\sqrt{-z}}{\sqrt{n}} j+
\left[\left(\frac{\alpha}{2}-\frac{1}{4}\right)j-\frac{zj^2}{2}\right]\frac{1}{n}+
\mathcal{O}(n^{-3/2}),
\end{equation}
again in accordance with \cite[\S 3, Lemma 1]{DHM1d-NA11}.
\end{itemize}

Higher order terms $U_m(\alpha,\beta,j,z)$ can be computed at a small extra cost. We omit the details for brevity, but the implementation of the procedure using symbolic computation is quite straightforward. 

\section*{Acknowledgements}
The authors are very grateful to Nico M. Temme for very useful discussions and extra information on the coefficients in the Perron expansion.

\end{document}